\documentclass {amsart}

\newtheorem{theorem}{Theorem}
\newtheorem{lemma}[theorem]{Lemma}
\newtheorem{corollary}[theorem]{Corollary}

\theoremstyle{remark}

\newcommand{\C}{\mathbb{C}}
\newcommand{\F}{\mathbb{F}}
\newcommand{\K}{\mathbb{K}}

\begin{document}

\title[On the Instability of the Riemann Hypothesis]
      {On the Instability of the Riemann Hypothesis over Finite Fields}
\date{}

%% First author

\author{P. M. Gauthier}

\address{D\'epartement de math\'ematiques et de statistique,
         Universit\'e de Montr\'eal,
         CP-6128 Centreville,
         Montr\'eal, H3C3J7,
         CANADA}

\email{gauthier@dms.umontreal.ca}

%% Second author

\author{N. Tarkhanov}

\address{Universit\"{a}t Potsdam,
         Institut f\"{u}r Mathematik,
         Postfach 60 15 53,
         14415 Potsdam,
         Germany}

\email{tarkhanov@math.uni-potsdam.de}

%%    General info



\dedicatory{}

\begin{abstract}
We show that it is possible to approximate the zeta-function of a curve over
a finite field by meromorphic functions which satisfy the same functional
equation and moreover satisfy (respectively do not satisfy) the analogue of
the Riemann hypothesis.
In the other direction, it is possible to approximate holomorphic functions by
simple manipulations of such a zeta-function.
We also consider the value distribution of zeta functions of function fields over finite fields from the viewpoint of Nevanlinna theory.
\end{abstract}

\maketitle

\tableofcontents

\section{Introduction}

The Riemann hypothesis concerns the Riemann zeta-function, but analogues of
the Riemann hypothesis have been formulated for other zeta-functions such as
zeta-functions of number fields and zeta functions of function fields.
There is no number field for which the Riemann hypothesis has been either
confirmed or disproved.
The only zeta-functions for which the Riemann hypothesis has been confirmed
are zeta-functions of function fields over finite fields
Of course, there is the hope of imitating the proof of the Riemann hypothesis
over function fields in order to prove the Riemann hypothesis for the Riemann
zeta-function, but this approach encounters serious obstacles.

The Riemann hypothesis for the Riemann zeta-function is unstable in the sense
that, in the vicinity of the Riemann zeta-function $\zeta (s)$, there are
functions (different from $\zeta$) which satisfy as well as functions which do
not satisfy an analogue of the Riemann hypothesis.
In the present paper we shall show that an analogous situation holds for
zeta-functions of curves over finite fields.
This portion of the study can be labeled
   ``approximation {\em of} zeta-functions of function fields over finite
     fields.''
We shall also investigate
   ``approximation {\em by} zeta-functions of function fields over finite
     fields''
and show that such zeta-functions have certain approximation properties
analogous to those of the Riemann zeta-function.
Namely, we shall show that all holomorphic functions can be approximated by
elementary manipulations of zeta-functions of function fields over finite
fields.

By a function field (of one variable) $\F$ over a field $\K$, we mean a
finitely generated field extension of $\K$ of transcendence degree $1$.
Equivalently, $\F$ is the function field of a smooth, irreducible, projective,
algebraic curve over $\K$.
F.K. Schmidt has defined the zeta-function $\zeta_\F(s)$ of the function field
$\F$ as follows.
\begin{equation}
\label{Schmidt}
   \zeta_\F (s)
 = \sum_\alpha \frac{1}{|\alpha|^s}
 = \prod_p \left( 1 - \frac{1}{|p|^s} \right)^{-1},
\end{equation}
where
   $\alpha$ ranges over the positive divisors,
   $p$ ranges over the prime divisors
and
   $|\alpha|$ is the absolute norm.
Since it would take too much room to explain this definition in detail and
since we shall use a different (but equivalent) definition, we refer the
interested reader to the files of Peter Roquette \cite{R}, where an excellent
historical account is given of the proof of the Riemann hypothesis for
zeta-functions of curves over finite fields.

Helmut Hasse first proved the analogue of the Riemann hypothesis for elliptic
curves over finite fields in 1934.
The case of general curves was obtained by Andr\'e Weil in 1942.
Over 30 years later, Pierre Deligne extended the Riemann hypothesis to
arbitrary varieties over finite fields.
In his description of the Riemann hypothesis as one of the millennium
problems, on the Clay Institute website, Enrico Bombieri ranks this as one of
the crowning achievements of 20th century mathematics.
Bombieri also writes that this is the best evidence in support of the Riemann
hypothesis.

It was observed by C.F. Osgood \cite{O} and further developed by
                   P. Vojta \cite{V}
that there is a formal analogy between certain aspects of number theory and
the value distribution theory of meromorphic functions (Nevanlinna theory).
In particular, M. Van Frankenhuijsen \cite{F} suggests that Nevanlinna theory
might be used to adapt the proof of the Riemann hypothesis for zeta-functions
of function fields over finite fields in order to obtain a proof of the
Riemann hypothesis for the Riemann zeta-function.

For the Riemann zeta-function, the basic notions of Nevanlinna theory were
studied only recently
   \cite{Y} (see also \cite{AG})
and more generally, J\"orn Steuding (\cite{S3}, \cite{S4} and \cite{S7}) has
studied the Nevanlinna theory for the Selberg class, but it seems that the
zeta-functions of function fields are not in the Selberg class.
One of the requirements to be in the Selberg class is that the function have a
pole at $1$ and nowhere else.
Steuding says this axiom is very important for Selberg class.
However, the zeta-functions of function fields over finite fields have many
poles.

In this paper, we study the basic notions of Nevanlinna theory for
zeta-functions of function fields over finite fields.
Then, we find analogues of results on approximation of the Riemann
zeta-function by functions which satisfy and by functions which do not satisfy
an analogue of the Riemann hypothesis.
In the other direction, we also show the possibility of approximating
arbitrary holomorphic functions by simple manipulations of zeta-functions
of function fields over finite fields.
Our motivation is to better understand the relation between the
Riemann zeta-function and zeta-functions of function fields over finite fields.
We find it interesting to see how many things are the same for zeta-functions of
function fields over finite fields and for Riemann's zeta-function.
This is a bit in the spirit of Bombieri and others that essential features of both objects should be the same although their analytic characters are rather different.
Whether or not this will actually lead to a better understanding of the
Riemann hypothesis, we believe the investigation opens a new line of research
of independent merit.

\section{Zeta functions of function fields over finite fields}

Henceforth, when we speak of a function field $\F$, it is understood that $\F$
is function field (of one variable) over a field $\K$ as defined in the
introduction and moreover that the base field $\K$ is finite.
A finite field $\K$ is uniquely determined by its size, which must be of the
form $q = p^r$, where $p$ is a rational prime and $r$ a natural number.
When the base field $\K$ is finite, the series and infinite product in
(\ref{Schmidt}) both converge for $\Re s > 1$.
Hence, the zeta-function $\zeta_\F (s)$ is a well defined holomorphic function
in the half-plane $\Re s > 1$.

Making the substitution $u = q^{-s}$, where $q$ is the order of the finite
field $\F$ and setting $Z (u) = \zeta_\F (s)$, it is known that $Z (u)$ is a
rational function.
In other words, $\zeta_\F (s)$ is a rational function of $q^{-s}$.
Since $Z (u)$ is rational, it is defined for all values of
$u \in \overline{\C}$, not just on the image $\{ u: u=q^{-s}, \Re s > 1 \}$.
Thus, $Z (u (s))$ is a meromorphic function on all of $\C$ which coincides
with the zeta-function $\zeta_\F (s)$ on the half-plane $\Re s > 1$.
Therefore $Z (u (s))$ is the (unique) meromorphic continuation of the
zeta-function $\zeta_\F (s)$ to the whole complex plane $\C$.

The zeta-function $\zeta_\F (s)$ satisfies the functional equation
\begin{equation}
\label{functional zeta}
   \zeta_\F (1-s) = q^{(g-1)(2s-1)} \zeta_\F (s),
\end{equation}
where $g$ is the genus.
Setting $Q (s)= q^{(g-1) s}$, for $f$ meromorphic on $\C$, we write
$\Lambda_\F (f,s) = Q(s) f(s)$.
The functional equation for $\zeta_\F$ can then be written in the form
$$
   \Lambda_\F (\zeta_\F, 1-s) = \Lambda_\F (\zeta_\F, s),
$$
which is symmetry with respect to the point $1/2$.
Since $\Lambda (\zeta_\F, s)$ is real for real $s$, the functional equation
can also be written
\begin{equation}
\label{1/2 symmetry}
   \overline{\Lambda_\F (\zeta_\F, 1-\overline{s})}
 = \Lambda_\F (\zeta_\F, s),
\end{equation}
which is the required form of functional equation for a function to belong to
Selberg class.

The rational function $Z (u)$ can be written in the form
$$
   Z (u) = \frac{L (u)}{(1-qu) (1-u)},
$$
where $L (u)$ is a polynomial of the form
$$
   L(u)
 = \sum_{j=0}^{2g} c_j u^j
 = 1 + (N-q-1) u + \cdots + q^g u^{2g}.
$$
If we represent $\F$ as the function field of a smooth curve $C$ then $N$ is the number of $\K\,$-rational points of $C$, that is, points of $C$ each of whose coordinates belong to $\K$, and all complex roots of $L (u)$ have norm $q^{-1/2}$ (which is ``the Riemann hypothesis'').
This is essentially what Weil showed (and then Deligne in a more general case).
Also, in some sense, almost any polynomial verifying this can appear (as the $L\,$-function of an Abelian variety this is really ``almost true,'' due essentially to Waterhouse, but to be the one corresponding to a curve is much more delicate, something studied by E. Nart and others in some recent papers).
Note that $Z (u)$ has simple poles at the points $u = 1/q$ and $u = 1$.

The coefficients of $L (u)$ satisfy the symmetry relation
\begin{equation}
\label{c_j}
   c_j = c_{2g-j} q^{j-g}.
\end{equation}
Thus,
$$
   L (u)
 = \sum_{j=0}^{2g} c_j u^j
 = 1 + (N-q-1) u + \cdots + q^{g-1} (N-q-1) u^{2g-1} + q^g u^{2g},
$$
and the zeta-function $\zeta_\F (s)$ has the representation
\begin{equation}
\label{zeta}
   \zeta_\F (s)
 = \frac{1+(N-q-1) u + \cdots + q^{g-1} (N-q-1) u^{2g-1} + q^g u^{2g}}
        {(1-qu)(1-u)}.
\end{equation}
Since the right side has poles at $u = 1$ and $u = 1/q$ and $u = q^{-s}$,
the zeta-function $\zeta_\F$ has simple poles at the corresponding points
$$
   \Re s = 1, \,
   \Im s = j \frac{2 \pi}{\log q},
   \,\,\,\,\,\,\,\,
   j = 0, \pm 1, \pm 2, \cdots
$$
and
$$
   \Re s=0, \,
   \Im s = j \frac{2 \pi}{\log q},
   \,\,\,\,\,\,\,\,
   j = 0, \pm 1, \pm 2, \cdots.
$$
Hence, the function $\zeta_\F$ has infinitely many poles on the lines
   $\Re s = 1$ and
   $\Re s = 0$.
In particular, $\zeta_\F$ has simple poles at $s = 1$ and $s = 0$.

The residue of $\zeta_\F (s)$ at the simple pole $s = 1$ is an important
number given by the formula
$$
   \lim_{s \rightarrow 1} (s-1) \zeta_\F (s)
 = \frac{q^{1-g} \cdot h}{(q-1) \log q}.
$$
This is the class number formula giving the residue of $\zeta_\F (s)$ at
$s = 1$ as a function of important invariants of the function field $\F$,
namely, the class number $h = h_\F$, of the function field, the genus $g$ of
the function field and the cardinality $q$ of the base field $\K$.
In terms of $L$ this yields the class number formula
$$
   L (1) = h.
$$

The symmetry relation (\ref{c_j}) is equivalent to the assertion that the
polynomial $L (u)$ satisfies the functional equation
\begin{equation}
\label{eq.fefL}
   L \left( \frac{1}{qu} \right) = q^{-g} u^{-2g} L (u).
\end{equation}
This also follows directly from the functional equation (\ref{functional zeta}) for $\zeta_\F$.
In fact, the functional equations (\ref{functional zeta}) and
                                  (\ref{eq.fefL})
are equivalent.
Moreover, writing $\mathcal{L} (u) = u^{-g} L(u)$, for $u \not=0$, these two
functional equations are also equivalent to the functional equation
\begin{equation}
\label{eq.efefL}
   \mathcal{L} \left( \frac{1}{q u} \right) = \mathcal{L} (u),
\end{equation}
which expresses symmetry with respect to holomorphic inversion with respect to
the critical circle $|u| = 1/\sqrt{q}$, corresponding to the critical axis
$\Re s = 1/2$, via the substitution $u = q^{-s}$.
Also, from the Dirichlet series, we see that $\zeta_\F (\sigma)$ is real, for
$\sigma$ real, $1 < \sigma < +\infty$.
Thus $\zeta_\F (\overline{s}) = \overline{\zeta_\F (s)}$, for $\Re s > 1$.
The same symmetry must hold for all $s$.
Hence,
\begin{equation}
\label{real symmetry}
   \zeta_\F (\overline{s})
 = \overline{\zeta_\F (s)}.
\end{equation}

It follows from this double symmetry (\ref{real symmetry}) and
                                     (\ref{1/2 symmetry})
that the zeros of $\zeta_\F$ are symmetric with respect to the real axis and
the point $1/2$ and from the Dirichlet series representation for $\zeta_\F$,
we see that $\zeta_\F$ has no zeros for $\Re s > 1$.
Thus, from the symmetry relations it follows that $\zeta_\F$, unlike the Riemann
zeta-function, has no zeros for $\Re s < 0$.
Hence, all zeros of the zeta-function $\zeta_\F (s)$ lie in the critical strip
$0 \le \Re s \le 1$, and they are symmetric with respect to the real axis and
the point $s = 1/2$.
To prove the associated Riemann hypothesis, then, it is sufficient to show
that $\zeta_\F (s)$ has no zeros in the open half-plane $\Re s > 1/2$.
Equivalently, it is sufficient to show that the function $Z (u)$ has no zeros
in the punctured open disc $0 < |u| < q^{-1/2}$.
At $u = 0$, the function $Z (u)$ has the value $1$.
Hence, the function
$$
   L (u) = (1-u) (1-qu) Z (u)
$$
also has the value $1$ at $u=0$ and, to show the Riemann hypothesis for
$\zeta_\F(s)$, it is sufficient to show that the function $L (u)$ has no zeros
in the punctured open disc $0 < |u| < q^{-1/2}$.
Equivalently, it is sufficient to show that the function $L^\prime/L$ has no
poles in $0 < |u| < q^{-1/2}$.
In a neighborhood of $u=0$,
$$
   \frac{L^\prime}{L}(u) = \frac{1}{u} \sum_{n=1}^\infty a_n u^n.
$$
It is sufficient to show that the radius of convergence of this series is at
least $q^{-1/2}$.
Thus, to prove the Riemann hypothesis, it is sufficient to show that
   $a_n = O (q^{n/2})$
as $n \rightarrow \infty$.
Weil showed that
$$
   |a_n| \leq 2g \cdot q^{n/2}
 \,\,\,\,\,\,\,\, (n=1, 2, \ldots)
$$
where $g$ is the genus.

%%%%%%%%%%%%%%%%%%%%%%%%%%%%%%%%%%%%%%%%%%%%%%%%%%%%%%%%%%%%%%%%%%%%%%%%%%%%%%%%%%%%%%%%%%%%%%%%%%%%%%%%%%%%%%%%%%%%%
%%%%%%%%%%%%%%%%%%%%%%%%%%%%%%%%%%%%%%%%%%%%%%%%%%%%%%%%%%%%%%%%%%%%%%%%%%%%%%%%%%%%%%%%%%%%%%%%%%%%%%%%%%%%%%%%%%%%%

\section{Nevanlinna theory}

Nevanlinna theory studies value distribution and growth of meromorphic functions.
If $\varphi (r)$ and
   $\psi (r)$
are real-valued functions defined for $r > 0$, we shall say that $\varphi$
is asymptotic to $\psi$ (as $r \to \infty$), denoted $\varphi \sim \psi$, if
$$
\lim_{r \to \infty} \frac{\varphi (r)}{\psi (r)} = 1.
$$

If $f$ is a function meromorphic on $\C$, we denote, as usual, the Nevanlinna
characteristic function of $f$ by
$$
   T (r,f) = m (r,f) + N (r,f),
 \,\,\,\,\,\,\,\,
   0 \leq r < \infty,
$$
where
   $m (r,f)$ is the proximity function and
   $N (r,f)$ is the integrated counting function for the value $\infty$.
For entire functions,
$$
   T (r,f)
 = m(r,f)
 = \frac{1}{2\pi} \int_0^{2\pi} \log^+ |f (r e^{i \theta})| d \theta.
$$
Thus, to calculate the characteristic functions of entire functions, the
following properties are useful:
$$
   \log^+ |a+b|
 \leq
   \log^+ (2 \max \{ |a|,|b| \})
 \leq
   \log^+ |a| + \log^+ |b| + \log 2,
$$
\begin{equation}
\label{log+(ab)}
   \log^+ |ab|
 \leq
   \log^+ |a| + \log^+ |b|,
\end{equation}
$$
   \log^+ (1/a) = \log^+ a - \log a.
$$
If $f_1, f_2, \ldots, f_n$ are meromorphic functions on $\C$, then
$$
   T \Big( r, \sum_{j=1}^n f_j \Big)
 \leq
   \sum_{j=1}^n T (r, f_j) + \log n,
$$
and
\begin{equation}
\label{T(fg)}
   T \Big( r, \prod_{j=1}^n T (r, f_j) \Big)
 \leq
   \sum_{j=1}^n T (r, f_j).
\end{equation}
If $f$ is meromorphic in $\C$ and about zero has the expansion
$$
   f (s) = \sum_{j=k}^\infty a_j z^j,
$$
with $a_k \not= 0$, then
$$
    T (r,f) = T (r,1/f) + \log |a_k|.
$$
More generally, the First Fundamental Theorem of Nevanlinna Theory states
that, if
   $\alpha$ is a finite value and
   the function $f$ is not identically equal to $\alpha$,
then, setting
   $T (r,\alpha,f) = T(r,1/(f-\alpha))$,
the characteristic function has the following property:
$$
   T (r,f) = T (r,\alpha,f) + O (1),
 \,\,\,\,\,\,\,\,
   \mbox{as} \,\,\,\,\,\,\,\, r \rightarrow \infty.
$$

Let us look more closely at the characteristic function for $\zeta_\F$,
$$
   T (r,\zeta_\F)
 = T (r, Z (u))
 = T \Big( r, \frac{L (u)}{(1-u) (1-qu} \Big).
$$
First, consider $f = 1/(1-u) (1-qu)$ as a function of $s$.
Then, $f$ has a simple pole at $s = 0$, so $k = -1$ and $a_{-1}$ is the
residue of $f$ at $0$, which is non-zero.
In fact, since the residue is invariant under change of chart, $a_{-1}$ is the
residue of $1/(1-u)(1-q u)$ as a function of $u$ at $u=1$, which is 1/(q-1), 
and so
$$
   T (r,1/(1-u)(1-qu)) = T (r,(1-u)(1-qu)) + \log (q-1).
$$
Applying the above properties of the characteristic function to the
zeta-function, we have the following:
\begin{eqnarray*}
   T (r,\zeta_\F)
 & = &
   T \Big( r,\frac{L (u)}{(u-1)(1-qu)} \Big)
\\
 & = &
   \max \Big( T (r,L (u)), T \Big( r,\frac{1}{(u-1)(1-qu)} \Big) \Big)
\\
 & = &
    \max (T (r,L (u)), T (r,(u-1)(1-qu)) + \log(q-1))
\end{eqnarray*}
whence
$$
    T (r,\zeta_\F)
 = \max (m (r,L (u)), m(r,(u-1)(1-qu)) + \log (q-1)).
$$

The order of a function $f$ meromorphic on $\C$ is given by
$$
   \rho = \limsup_{r \rightarrow \infty} \frac{\log T (r,f)}{\log r}
$$
and the lower order is given by
$$
    \underline{\rho}
 = \liminf_{r \rightarrow \infty} \frac{\log T (r,f)}{\log r}.
$$
If $f$ is an entire function, then
$$
   \rho
 = \limsup_{r \rightarrow \infty}
   \frac{\log m (r,f)}{\log r}
$$
and
$$
   \underline{\rho}
 = \liminf_{r \rightarrow \infty}
   \frac{\log m (r,f)}{\log r}.
$$

The properties of order expressed in the following theorem are well known.

\begin{theorem}
\label{t.tpoo}
\begin{equation}
\label{rho(1/f)}
   \rho_f = \rho_{1/f},
 \,\,
   \mbox{if} \,\, 1/f  \,\, \mbox{is defined.}
\end{equation}
\begin{equation}
\label{rho(f+g)}
   \rho_{f+g} \leq \max (\rho_f,\rho_g).
\end{equation}
\begin{equation}
\label{rho(fg)}
  \rho_{fg} \leq \max (\rho_f,\rho_g).
\end{equation}
Moreover, if $\rho_f < \rho_g$, then the last two inequalities become
equalities.
\end{theorem}

From the properties we have listed for the characteristic function, one can show the following.

\begin{theorem}
\label{t.ihinc}
Suppose $h (u)$ is a non-constant rational function and $u = q^{-s}$.
Then,
   $\underline{\rho}_h = \rho_h = 1$,
as a function of $s$.
\end{theorem}

The following particular case is worth stating as a theorem.

\begin{theorem}
For the zeta-function $\zeta_\F$ over a finite field $\F_q$, both the order
and the lower order are $1$.
\end{theorem}

\begin{proof}
By formula (\ref{zeta}), the zeta-function $\zeta_\F$ is a non-constant
rational function of $u = q^{-s}$.
\end{proof}

Let $f$ be meromorphic on $\C$.
The Nevanlinna defect (or deficiency) of $f$ for a value
$\alpha \in \overline{\C}$ is
$$
   \delta (\alpha,f)
 = 1 - \limsup_{r \rightarrow \infty} \frac{N (r,\alpha,f)}{T (r,f)},
$$
and $\alpha$ is said to be a deficient value for $f$ if the deficiency
$\delta (\alpha,f)$ is not zero.
The Second Fundamental Theorem of Nevanlinna Theory asserts that
$$
   \sum_{\alpha \in \overline{\C}} \delta (\alpha,f) \leq 2.
$$

Let us now count the zeros and, more generally, the $\alpha$-values of the
zeta-function $\zeta_\F (s)$.
We shall use the property that $u = q^{-s}$ is periodic in $s$ with period
$2 \pi i/\log q$ and that it is injective in each period strip.
Thus, if $Z (u) = \alpha$ has $k$ solutions $u \in \C$ (counting multiplicity),
then
$$
   n (r,\alpha,\zeta_\F) \sim \frac{k \log q}{2 \pi} r
$$
and so also
$$
    N (r,\alpha,\zeta_\F) \sim \frac{k \log q}{2 \pi} r.
$$
Thus, the Nevanlinna defect for the value $\alpha$
$$
   \delta (\alpha,\zeta_\F)
 = 1 - \limsup_{r \rightarrow \infty}
       \frac{N (r,\alpha,\zeta_\F)}{T (r,\zeta_\F)}
$$
is the same, for all values $\alpha \not= Z (\infty)$.
By the Second Fundamental Theorem,
$$
   \sum_{\alpha \in \overline{\C}} \delta (\alpha,\zeta_\F)
 \leq
   2,
$$
and so
   $\delta (\alpha,\zeta_\F) = 0$ for all values $\alpha \not= Z (\infty)$.

From formula (\ref{zeta}), we see that if the genus $g$ is zero, then
$$
   \zeta_\F (s) = Z (u) = \frac{1}{(1-qu)(1-u)},
 \,\,\,\,\,\,\,\,
   \mbox{with} \,\,\,\,\,\,\,\,  u = q^{-s}.
$$
The function $Z (u)$ has no finite zeros and so the zeta-function has no
zeros.
This certainly confirms the Riemann hypothesis in this case.
$Z (u)$ assumes every non-zero value $\alpha$, including $\infty$ twice
(counting multiplicity) in $\C$.
Consequently, $\delta (\alpha,\zeta_\F) = 0$, for all non-zero values $\alpha$.
Since $\zeta_\F$ has no zeros, of course $\delta (0,\zeta_\F) = 1$ and $0$ is
a totally deficient value.

If the genus $g$ is $1$ (elliptic curves), then
$$
   \zeta_\F (s) = Z (u) = \frac{1 + (N-q-1) u + q u^2}{(1-qu)(1-u)},
 \,\,\,\,\,\,\,\,
    \mbox{with} \,\,\,\,\,\,\,\,  u = q^{-s}.
$$
The function $Z (u)$ takes the value $1$ with multiplicity $1$ at $\infty$ and
also at $0$.
Thus, $Z (u)$ assumes every value of $\overline{\C}$ other than $1$ the same
number of times in $\C$.
As before, it follows that the deficiencies $\delta (\alpha,\zeta_\F)$ are
zero, for all values of $\overline{\C}$ different from $1$.
Consider the value $1$.
Since $Z (u)$ assumes the value $1$ only once in $\C$, the function
$N (r,1,\zeta_\F)$ is asymptotic to $r \log q / 2 \pi$.
Similarly, since the value $0$ is assumed twice by $Z (u)$ in $\C$, the
function $N (r,0,\zeta_\F)$ is asymptotic to $2 r \log q / 2 \pi$.
Thus,
\begin{eqnarray*}
   \delta (1,\zeta_\F)
 & = &
   1 - \limsup_{r \rightarrow \infty}
       \frac{N (r,1,\zeta_\F)}{T (r,\zeta_\F)}
\\
 & = &
   1 - \frac{1}{2} \limsup_{r \rightarrow \infty}
       \frac{N (r,0,\zeta_\F)}{T (r,\zeta_\F)}
\\
 & = &
   \frac{1}{2}
 + \frac{1}{2}
   \Big( 1 - \limsup_{r \rightarrow \infty} \frac{N (r,0,\zeta_\F)}
                                                 {T (r,\zeta_\F)} \Big)
\\
 & = &
   \frac{1}{2} + \frac{1}{2} \delta (0,\zeta_\F)
\\
 & = &
   \frac{1}{2}.
\end{eqnarray*}
Hence, in the elliptic case ($g = 1$), the only deficient value is $1$ and its
deficiency is $1/2$.
We have $\delta (\alpha,\zeta_\F) = 0$ for all other values $\alpha$ of
$\overline{\C}$.
It is interesting that the value $1$ plays such a special role among all
values of $\overline{\C}$.
We note that, in the case of $L$-functions, Steuding \cite{S4} shows that
$\delta (\alpha,\zeta) = 0$, for all {\em finite} values $\alpha$.
For the Riemann zeta-function, Ye \cite{Y} had shown that there are no finite
deficient values (see also \cite{AG}).

If $g > 1$, then the function $Z (u)$ has $2 g$ finite $\alpha$-points, for
each finite value $\alpha$.
Hence, the Nevanlinna integrated counting function $N (r,\alpha,\zeta_\F)$ is
asymptotic to $2 g r \log q / 2 \pi$.
The integrated counting function for the poles $N (r,\zeta_\F)$ is also
asymptotic to $2 g r \log q/2\pi$.
Thus, as before, $\delta (\alpha,\zeta_\F) = 0$ for all finite values $\alpha$,
while
\begin{eqnarray*}
   \delta (\infty,\zeta_\F)
 & = &
   1 - \limsup_{r \rightarrow \infty} \frac{N (r,\zeta_\F)}{T (r,\zeta_\F)}
\\
 & = &
   1
 - \frac{1}{g} \limsup_{r \rightarrow \infty} \frac{N (r,0,\zeta_\F)}
                                                   {T (r,\zeta_\F)}
\\
 & = &
   \frac{g-1}{g}
 + \frac{1}{g}
   \Big( 1 - \limsup_{r \rightarrow \infty} \frac{N (r,0,\zeta_\F)}
                                                 {T (r,\zeta_\F)} \Big)
\\
 & = &
   \frac{g-1}{g} + \frac{1}{g} \delta (0,\zeta_\F)
\\
 & = &
   \frac{g-1}{g}.
\end{eqnarray*}

We summarize these calculations.

\begin{theorem}
Let $g$ be the genus of the function field.
$$
   \begin{array}{llllllll}
     \mbox{If}
   & g=0,
   & \mbox{then}
   & \delta (\alpha,\zeta_\F) = 0
   & \mbox{for all}
   & \alpha \not= 0,
   & \mbox{and}
   & \delta (0,\zeta_\F) = 1.
\\
     \mbox{If}
   & g=1,
   & \mbox{then}
   & \delta (\alpha,\zeta_\F) = 0
   & \mbox{for all}
   & \alpha \not= 1,
   & \mbox{and}
   & \delta (1,\zeta_\F) = 1/2.
\\
     \mbox{If}
   & g > 1,
   & \mbox{then}
   & \delta (\alpha,\zeta_\F) = 0
   & \mbox{for all}
   & \alpha \not= \infty,
   & \mbox{and}
   & \delta( \infty,\zeta_\F) = (g-1)/g.
   \end{array}
$$
\end{theorem}

We have verified that the zeta-function is of order $1$, but it would be
useful to have a more precise estimate for the growth of $\zeta_\F$.
Namely, we would like to know the type of $\zeta_\F$.
Recall that for a meromorphic function $f$ of finite non-zero order $\rho$,
the type $\lambda$ of $f$ is defined as
$$
   \lambda = \limsup_{r \rightarrow \infty} \frac{T (r,f)}{r^\rho}.
$$
Then, $f$ is said to be of maximum, mean, or minimum type according as the
type $\lambda_f$ is infinite, finite and positive, or zero respectively.
Let us denote the type of the zeta-function $\zeta_\F(s)$ by
   $\lambda_{\zeta_\F}$.

\begin{theorem}
The zeta-function $\zeta_\F$ is of order $1$.
If the genus $g$ is positive, then  $\zeta_\F$ is of mean type  $g \log q/\pi$.
If the genus $g$ is $0$, then $\zeta_\F$ is of mean type  $\log q/\pi$.
\end{theorem}

\begin{proof}
We have already shown that the zeta-function is of order $1$.
By the First Fundamental Theorem,
$$
   \lambda_{\zeta_\F}
 = \limsup_{r \rightarrow \infty} \frac{T (r,\zeta_\F)}{r}
 = \limsup_{r \rightarrow \infty} \frac{T (r,2,\zeta_\F)}{r}.
$$
If $\Gamma$ is of genus $g > 0$, then, since $2$ is neither $0,1$ nor
                                                            $\infty$,
we have shown that $N (r,2,\zeta_\F)$ is asymptotic to $g r \log q/\pi$.
Thus,
$$
   \limsup_{r \rightarrow \infty} \frac{T (r,2,\zeta_\F)}{r}
 \geq
   \limsup_{r \rightarrow \infty} \frac{N (r,2,\zeta_\F)}{r}
 = \frac{g \log q}{\pi} r.
$$
Thus, $\lambda_{\zeta_\F} \geq g \log q/\pi$, if $g > 0$.

If $g = 0$, we deduce that $N (r,2,\zeta_\F)$ is asymptotic to $r \log q/\pi$, and so we get $\lambda_{\zeta_\F} \geq \log q/\pi$.

Now let us show the opposite inequality.
By the First Fundamental Theorem,
$$
   \lambda_{\zeta_\F}
 = \limsup_{r \rightarrow \infty} \frac{T (r,\alpha,\zeta_\F)}{r}
 \geq
   \limsup_{r \rightarrow \infty} \frac{N (r,\alpha,\zeta_\F)}{r}.
$$
If $g > 0$, then $N (r,\alpha,\zeta_\F)$ is asymptotic to $g r \log q/\pi$ for
all $\alpha$ different from $\infty$ and $1$.
Thus, $\lambda_{\zeta_\F} \geq g \log q/\pi$.
If $g = 0$, then $N (r,\alpha,\zeta_\F)$ is asymptotic to $r \log q/\pi$ for
all $\alpha \not= 0$.
Thus, $\lambda_{\zeta_\F} \geq \log q/\pi$.
\end{proof}

\section{Approximation by zeta functions}

For a subset $S \subset \C$, we denote by $\mathcal{O} (S)$ the set of
functions $f$ such that $f$ is holomorphic on an open neighborhood (depending
on $f$) of $S$.
If $f$ is holomorphic on an open neighborhood $U$ of $S$ and
   $g$ is holomorphic on an open neighborhood $V$ of $S$ and
$f = g$ on some open neighborhood of $S$ contained in $U \cap V$, then we
consider $f$ and $g$ to be the same element of $\mathcal{O} (S)$.

\begin{theorem}
For each compact subset $K$ of $\C$,
for each function $f \in \mathcal{O} (K)$ and
for each $\epsilon > 0$, there are finitely many values $a_k$,
                                                        $b_k$ and
                                                        $\lambda_k$,
$k = 1, \ldots, n$, such that
$$
   \Big| f (s) - \sum_{k=1}^n \lambda_k \zeta_\F (a_k s + b_k) \Big|
 < \epsilon
 \ \ \mbox{for}\ \ s \in K.
$$
\end{theorem}

For the Riemann zeta-function, the authors have shown a stronger result
\cite{GT}.

\begin{proof}
Suppose we are given a compact subset $K$ of $\C$,
                     a function $f \in \mathcal{O} (K)$ and
                     positive $\epsilon$.
The function $f$ is holomorphic in some bounded open set $U$ containing $K$.
By multiplying $f$ by a smooth function $\chi$ with
   $\mathrm{supp}\, \chi \subset U$ and
   $\chi = 1$ on a neighborhood of $K$,
we may assume that $f$ itself is smoothly defined on all of $\C$.
We note that the compact sets $\mathrm{supp}\, f$ and
                              $\mathrm{supp}\, \bar{\partial} f$
are both contained in $U$ and
   $\mathrm{supp}\, \bar{\partial} f$ is disjoint from $K$.

For $\eta \not= 0$ sufficiently small, the non-zero poles
   $p_j$, $j = 1, 2, \dots$,
of the function $\zeta_\F (\eta s)$ lie outside of the bounded set $U$ and
also outside of the bounded set $U - U = \{ s-z: s, z \in U \}$.
Since the pole of $\zeta_\F (\eta s)$ at zero is simple, we may write
   $\zeta_\F (\eta s)$
in the form
$$
   \zeta_\F (\eta s) = \frac{a}{\pi s} + h (s),
$$
where
   $h$ is a meromorphic function on $\C$ all of whose poles lie outside of the
   bounded sets $U$ and $U - U$.
Since, in fact, all of the poles of the function $\zeta_\F (\eta s)$ are
simple, $\zeta_\F (\eta s)$ is locally integrable and may be considered as a
distribution $T_\zeta$.
Noting that $(\pi s)^{-1}$ is a fundamental solution, which we denote by
$\Phi$, for the partial differential operator $\bar{\partial}$, we have
   $T_\zeta = a \Phi + h$,
as distributions.

Since $f \in C_0^\infty (U)$, we have the representation
\begin{eqnarray*}
   f (s)
 & = &
   (\bar{\partial} f * \Phi) (s)
\\
 & = &
   \int \!\! \int (\bar{\partial} f) (z) \Phi (s-z) dxdy
\\
 & = &
   a^{-1}
   \int \!\! \int (\bar{\partial} f) (z) \zeta_\F (\eta s - \eta z) dxdy
 - a^{-1}
   \int \!\! \int (\bar{\partial} f) (z) h (s-z) dxdy.
\end{eqnarray*}
Consider the second integral.
\begin{eqnarray*}
   \int \!\! \int (\bar{\partial} f) (z) h (s-z) dxdy
 & = &
 - \int \!\! \int f (z) \bar{\partial}_z h (s-z) dxdy
\\
 & = &
 - \int \!\! \int_U f (z) \bar{\partial}_z h (s-z) dxdy,
\end{eqnarray*}
because $\mathrm{supp}\, f \subset U$.
Moreover, since $h (s-z)$ is holomorphic in $U \times U$, we have
   $\bar{\partial}_z h (s-z) = 0$,
for $s, z \in U$.
Thus, for $s \in U$
$$
    f (s)
 = a^{-1} \int \!\! \int (\bar{\partial} f) (z) \zeta_\F (\eta (s-z)) dxdy
 = a^{-1} \int \!\! \int_{\mathrm{supp}\, \bar{\partial} f}
                         (\bar{\partial} f) (z) \zeta_\F (\eta (s-z)) dxdy.
$$
In particular, this formula holds for $s \in K \subset U$.
For
   $(s,z) \in K \times \mathrm{supp}\, \bar{\partial} f$,
the point $s-z$ is in $U-U$ and so, by the choice of $\eta$, it follows that
   $\eta (s-z)$ is not a pole of $\zeta_\F$.
So the integrand is smooth for
   $(s,z) \in K \times \mathrm{supp}\, \bar{\partial} f$.
In particular, it is continuous and so we may approximate $f (s)$ by Riemann
sums of this integral.
These can be written in the form
$$
   \sum_{k=1}^n \lambda_k \zeta_\F (\eta s - \eta z_k),
$$
for $s \in K$.
In fact, since the integrand is uniformly continuous on
   $K \times \mathrm{supp}\, \bar{\partial} f$,
we may approximate $f$ within $\epsilon$ uniformly on $K$ by such Riemann
sums.
\end{proof}

\section{Instability of the Riemann hypothesis}

Let us say that a function $f$ meromorphic on $\C$ satisfies (an analogue of)
the ``Riemann Hypothesis'' if $f$ has no non-trivial zeros off the critical
axis.
Similarly, let us say that $f$ fails to satisfy (an analogue of) the
``Riemann hypothesis'' if it {\em does} have non-trivial zeros off the
critical axis.
The instability of the Riemann hypothesis refers to the  phenomenon, that near
the Riemann zeta-function, there are functions
   which do satisfy the ``Riemann hypothesis''
and functions
   which do not satisfy the ``Riemann hypothesis.''
This phenomenon was investigated, for example, in \cite{GZ},
                                                  \cite{GX} and
                                                  \cite{G}.
The intention was to show that this instability holds for many important
$L$-functions, including the Riemann zeta-function.
In this section we wish to point out that such instability also holds for the
Riemann hypothesis for zeta-functions of function fields over finite fields.

Let $\mathcal{M}$ be the space of meromorphic functions on $\C$ with the
topology of uniform convergence on compacta.
In this space, a sequence $g_n$ converges to $g$ if, on each compact set $K$
the functions $g_n$ eventually have the same poles with the same principal
parts as $g$ and $g_n - g$ tends to zero.
This is a complete metric space and hence of Baire category II.

Let $\mathcal{M}_\F$ be the class of functions in $\mathcal{M}$ sharing the
following properties with $\zeta_\F$:

i)
$$
   f (s) = \frac{h (u)}{(1-u)(1-qu)}, \, u=q^{-s},
$$
where $h$ is holomorphic on $\C^*$ and $h (1) = L (1)$,
                                       $h (1/q) = L (1/q)$.

ii)
The function $f$ satisfies the functional equation for $\zeta_\F:$
$$
    \mathcal L (u,h) \equiv u^{-g} h (u)
 =  \mathcal L \Big( \frac{1}{qu}, h \Big).
$$

iii)
$f (s) = \overline {f (\overline{s})}$.

\noindent
Note that from i) it follows that functions in $\mathcal{M}_\F$ have the same
poles with same principal parts as $\zeta_\F$.
Moreover, the zeros of $h$ are symmetric with respect to the real axis and the
critical circle.

It is important to emphasize that the functions in $\mathcal{M}_\F$ satisfy the
same functional equation (\ref{functional zeta}) as the zeta-function, for the
functional equations (\ref{functional zeta}),
                     (\ref{eq.fefL}) and
                     (\ref{eq.efefL})
are equivalent, not only for $\zeta_\F$, but for any function. 

Let $\mathcal{R}_\F$ be the set of those functions in $\mathcal{M}_\F$ which are
rational as functions of $u$.
More precisely,
$$
   f (s) = \frac{R (u)}{(1-u)(1-qu)}, \, u=q^{-s},
$$
where $R$ is rational with no poles on $\C \setminus \{ 0 \}$ and 
   $R (1) = L (1)$,
   $R (1/q) = L (1/q)$.
The functions in $\mathcal{R}_\F$ resemble the zeta function $\zeta_\F$ even
more than those in $\mathcal{M}_\F$.
Moreover, by Theorem \ref{t.ihinc}, functions in $\mathcal{R}_\F$ have the same
order $1$ as $\zeta_\F$.

\begin{lemma}
\label{nu}
Let $\nu (u)$ be a holomorphic function on $\C \setminus \{ 0 \}$, satisfying the relations
   $\nu (u) = \nu (1/qu)$ and
   $\nu (u) = \overline{\nu (\overline{u})}$,
and such that $\nu (1) = 1 = \nu (1/q)$.
Then for each $f \in \mathcal{M}_\F$, the function $\nu (u (s)) f(s)$
is also in $\mathcal{M}_\F$.
If, moreover, $\nu (u)$ is rational, then for each $f \in \mathcal{R}_\F$, the
function $\nu (u (s)) f(s)$ is also in $\mathcal{R}_\F$.
\end{lemma}

The following Walsh-type lemma on simultaneous approximation and interpolation
is due to Frank Deutsch \cite{D}.

\begin{lemma}
\label{walsh}
Given a locally convex complex vector space $X$ and
      a dense subspace $Y$ of $X$,
if
   $x \in X$, and
   $U$ is a neighborhood  of $0$, and
   $L_1, \ldots, L_n$ are finitely many continuous linear functionals on $X$,
then, there exists an element $y \in Y$ which simultaneously approximates and
interpolates $x$ in the sense that
   $y \in x+U$ and
   $L_j y = L_j x$, $j = 1, \ldots, n$.
\end{lemma}

Let $\mathcal{R}_\F^-$ be the subclass of $\mathcal{R}_\F$ for which the
``Riemann hypothesis''  fails.

\begin{theorem}
The class $\mathcal{R}_\F^-$ of functions in $\mathcal{R}_\F$ which fail to
satisfy the ``Riemann hypothesis'' form an open dense subfamily of
$\mathcal{R}_\F$.
\end{theorem}

\begin{proof}
Let $f \in \mathcal{R}_\F$, let $K$ be a compact subset of the complex
$s$-plane $\C_s$ and let $\alpha > 0$.
For $q < r < \infty$, let
$$
   A = \Big\{ \frac{1}{q r} \leq |u| \leq r \Big\}
$$
be an annulus in the complex $u$-plane $\C_u$, with $r$ so large that
$u (K) \subset A$, where $u = q^{-s}$.
Choose a point $u_0 \not= 0$ outside $A$.
By Runge's theorem, rational functions having no poles in the punctured plane
$\C^*$ are dense in the space of functions holomorphic on $A \cup \{ u_0 \}$.
Moreover, by the Walsh Lemma \ref{walsh}, we may not only approximate but also
interpolate at finitely many points.
Thus, there is a rational function $p_\delta$, having no poles on 
$\C \setminus \{ 0 \}$, such that 
   $|1 - p_\delta| < \alpha$ on $A$ 
and 
   $p_\delta$ takes the value $0$ at $u_0$ and 
                    the value $1$ at the points $u = 1$ and $u = 1/q$.
Set
$$
   \nu_\delta (u)
 = p_\delta (u)
   \Big( p_\delta \Big( \frac{1}{qu} \Big) \Big)
   \overline{p_\delta (\overline{u})}
   \Big( \overline{p_\delta \Big( \frac{1}{\overline{qu}} \Big)} \Big).
$$
Since $A$ is invariant under conjugation and inversion in the critical circle
$|u| = 1/\sqrt q$, given any $\alpha > 0$, we may choose $\delta$ so small
that $|1 - \nu_\delta| < \alpha$ on $A$.
Set $f^- (s) = \nu_\delta (u(s)) f(s)$.
By the lemma, $f^- \in \mathcal{R}_\F$.
From the definition of the class $\mathcal{R}_\F$, we may write
   $f (s) = \Phi (u)$,
with $u = q^{-s}$.
Let $M = \max |\Phi|$ on $\partial A$.
Choose $\alpha < \epsilon/M$.
On $\partial A$,  we have
   $|\nu_\delta \Phi - \Phi| = |\nu_\delta - 1| |\Phi| < \epsilon$.
But $\nu_\delta \Phi - \Phi$ is holomorphic on $\C \setminus \{ 0 \}$, so the same inequality holds on all of $A$ by the maximum principle.
Since $u (K) \subset A$, we have
   $|f^- (s) - f (s)| = |\nu_\delta (u) \Phi (u) - \Phi (u)| < \epsilon$
on $K$.
We have approximated $f$ by a function $f^-$ in $\mathcal{R}_\F$ which fails
to satisfy the ``Riemann hypothesis.''
Thus, the functions in $\mathcal{R}_\F$ which fail to satisfy the ``Riemann
hypothesis'' form a dense subclass of $\mathcal{R}_\F$.
That the family of such functions is open in $\mathcal{R}_\F$ follows
immediately from Rouch\'e's theorem.

It is possible to insure that the approximating functions are different from
$f$, because the lemmas employed allow much freedom in the construction.
Thus, the approximation is not trivial.
\end{proof}

\begin{corollary}
For every $f \in \mathcal{R}_\F$ (and in particular for $\zeta_\F$), there is
a sequence of functions $\{ f_n \}$ in $\mathcal{R}_\F$, $f_n \not= f$, which
fail to satisfy the ``Riemann hypothesis'' and for every $j = 0, 1, 2, \ldots$,
$$
   f_n^{(j)} \rightarrow f^{(j)},
 \,\,\,\, \mbox{as} \,\,\,\, n \rightarrow \infty.
$$
\end{corollary}

\begin{proof}
For holomorphic functions, uniform convergence on compacta implies uniform
convergence of all derivatives.
We have $(f_n - f) \rightarrow 0$ and so
        $(f_n^{(j)} - f^{(j)} \rightarrow 0$,
for all $j = 0, 1, 2, \ldots$.
Thus, $f_n^{(j)} \rightarrow f^{(j)}$, at all points $s$, where the functions
are holomorphic.
At the poles, the convergence
   $(f_n^{(j)} - f^{(j)}) \rightarrow 0$
can be interpreted as meaning that the Laurent coefficients of $f_n$ converge
to those of $f$.
\end{proof}

The preceding theorem asserts that ``most'' functions in the class
$\mathcal{R}_\F$ of rational functions ``resembling'' the zeta-function
$\zeta_\F$ fail to satisfy the ``Riemann hypothesis.''
An analogous result had been shown earlier for the Riemann zeta-function,
with the striking difference that $\zeta_\F$ is known to satisfy the Riemann
hypothesis.

In 1921, H. Hamburger showed that the functional equation for the Riemann
zeta-function $\zeta$ characterizes it completely in a certain sense.
Namely, he showed that $\zeta$ is unique among Dirichlet series, converging
for $\Re s > 1$, extending to the complex plane $\C$ as meromorphic functions
of finite order having only finitely many poles and satisfying the functional
equation.
Theorems such as the preceding one for $\zeta$ (see for example \cite{G}) show
that there are many other functions than $\zeta$ which satisfy the same
functional equation, but these examples are surely not of finite order.
In seeking for an analogue of Hamburger's theorem for zeta-functions over
finite fields, since these have infinitely many poles, it is natural to
replace the hypothesis that a function $f$ have only finitely many poles by
the hypothesis that $f$ have the same poles as $\zeta_\F$ and with the same
principal parts.
The preceding theorem not only gives many such functions satisfying the same
functional equation as $\zeta_\F$ - these functions are even of finite order
and, in fact, of order $1$ as is $\zeta_\F$.

Having approximated $\zeta_\F$ by similar functions which fail to satisfy the
``Riemann hypothesis,'' we now turn to approximating functions, and in
particular $\zeta_\F$, by functions (different from $\zeta_\F$) which {\em do}
satisfy the ``Riemann hypothesis.''

The following lemma is Theorem 40 in \cite{G} except that in Theorem 40, there
is only one $\beta$.
The proof for two points $\beta_1$ and $\beta_2$ is the same.

\begin{lemma}
\label{40}
Let $X$ be a set of uniform approximation in $\C$, let $\beta_1$, $\beta_2$ be
points of $X$ and let $Z = \{ z_1, z_2, \ldots \}$ be a discrete set in
$\C \setminus X$.
Suppose $\Phi$ is meromorphic on $\C$ and has zeros of respective orders $k_j$
at the points $z_j $.
Then, for each $\epsilon > 0$, and each sequence $\{ c_j \}$ of non-zero
values, there is an entire function $g$, taking the value $1$ at $\beta_1$ and
$\beta_2$ such that, on $X$,
$
    |1 - g| < \epsilon/2.
$
Moreover, $g$ has no zeros except at the points $z_j$, where
$$
   g^{(k)} (z_j) = 0, \, k = 0, \ldots, k_j-1,
 \,\,\,\,\,\,\,\, \mbox{and} \,\,\,\,\,\,\,\,
   g^{(k_j)} (z_j) = c_j.
$$
Hence, the function $\Phi/g$ approximates $\Phi$ on $X$, has the same value at
$\beta_1$ and $\beta_2$ and has the same zeros as $\Phi$ except for the points
$z_j$, where $\Phi/g$ takes the value $1/c_j$.
\end{lemma}

Given a function $f \in \mathcal{M}_\F$, we wish to construct an increasing
sequence of closed subsets $\{ E_n \}$ of the $s$-plane $\C_s$ on which we
shall approximate $f$.
But first we construct a sequence of compact sets $\{ K_n \}$ in the $u$-plane
$C_u$.
Let $r_0 = 1/\sqrt q$ and $r_1 < r_2 < \ldots$ be an increasing sequence with
$q < r_1$ and $r_n \rightarrow \infty$.
For $n = 1, 2, \dots$, put
$$
\begin{array}{ll}
   A_n = \{ u: r_{n-1} \leq |u| \leq r_n \},
 & B_n = A_1 \cup \ldots \cup A_n,
\\
   Z_n = \{ z: f (z) = 0, z \in B_n \},
 & \displaystyle
   U_n = \bigcup_{z \in Z_n} \{ u: \Re z < \Re u < \Re z+1/n, |u| > |z|-1 \},
\\
   D_{z,n} = \{ u: |u-z| < 1/n \},
 & \displaystyle
   D_n = \bigcup_{z \in Z_n} D_{z,n},
\end{array}
\!\!\!
$$
and
$
   K_n^+ = B_n \setminus (U_n \cup D_n).
$

Write $f (s) = \Phi (u)$.
Set $X_n^+ = \{ |u| \leq 1/\sqrt q \} \cup K_n^+$.
Then, $X_n^+$ is a compact subset of $\C$ with connected complement.
By Mergelyan's theorem, $X_n^+$ is a set of uniform approximation and so by
Lemma \ref{40}, for each $\epsilon > 0$, there is an entire function $g^+$
taking the value $1$ at the points $1$ and $1/q$, having no zeros except at
the real zeros $z$ of $\Phi$, outside the critical circle, where $g^+$ has
zeros of the same multiplicity as $\Phi$.
Moreover, $|1 - g^+| < \epsilon$ on $X_n^+$.

We now take care of real zeros of $\Phi$ inside the critical circle.
Set
$$
\begin{array}{rcl}
   K_n^- & = & \{ u: 1/qu \in K_n^+ \},
\\
   X_n^- & = & \{ u: 1/qu \in X_n^+ \},
\end{array}
$$
and $g^- (u) = g_n^+ (1/qu)$.
Then, $g^-$ is holomorphic in $\overline{\C} \setminus \{ 0 \}$, has the same
poles as $\Phi$ with the same principal parts, such that
   $|1 - g^-| < \epsilon_n$
on $X_n^-$ and $g^-$ takes real or infinite values on the real axis.
Moreover, $g^-$ has no zeros except at the real zeros of $\Phi$ inside the
critical circle, where it has zeros of the same multiplicity as $\Phi$.

Since the number $\epsilon_n$ is arbitrarily small, we replace the conclusion
   $|1 - g^\pm| < \epsilon_n$
by
   $|1 - 1/g^\pm| < \epsilon_n$,
from which it follows that
   $|\Phi - \Phi/g^\pm| < \epsilon_n |\Phi|$
on $X_n^\pm$ respectively.
Set $X_n = X_n^+ \cap X_n^- = K_n = K_n^+ \cup K_n^-$.
Since $g_n^\pm$ can be chosen to approximate $1$ arbitrarily well on
$X_n^\pm$, we may assume that
$$
   |\Phi - \frac{\Phi}{g^+ g^-}| < \epsilon_n |\Phi|
$$
on $X_n = K_n$.
Let $M = \max |\Phi|$ on $\partial K_n$.
Given $\epsilon > 0$, we may choose $\epsilon_n < \epsilon/M$.
Then,
$$
   |\Phi - \frac{\Phi}{g^+ g^-}| < \epsilon
$$
on $\partial K_n$.
Since $\Phi - \Phi/(g^+ g^-)$ is holomorphic in $\C \setminus \{ 0 \}$, we 
have the same inequality on all of $K_n$ by the maximum principle.

Since all geometric figures employed are symmetric with respect to the real
axis, we may assume (in all lemmas and proofs leading up to this point) that
$g^\pm$ take real or infinite values on the real axis.
Set $\mu = g^+ g^-$.
Then
   $\mu (u) = \overline{\mu (\overline{u})}$ and
   $\mu (u) = \mu (1/qu)$.
Hence, by Lemma \ref{nu},
   $\mu (u (s) f (s) \in \mathcal{M}_\F$.
We have shown that, given $f \in \mathcal{M}_\F$ with $f (s) = \Phi (u)$, we
can associate 
   an increasing sequence of compact sets $K_n$ in $\C \setminus \{ 0 \}$, 
   with $K_n \rightarrow \C \setminus \{ 0 \}$, 
such that,
   for any positive sequence $\{ \epsilon_n \}$,
there are functions $f_n \in \mathcal{M}_\F$ such that, setting
   $f_n (s) = \Phi_n (u)$,
we have
   $|\Phi_n - \Phi| < \epsilon_n$ on $K_n$ and
   $\Phi_n$ has no real zeros off the critical circle.

\begin{theorem}
Given a function $f \in\mathcal{M}_\F$, there exists an increasing sequence of
closed sets $E_n \nearrow \C$, such that for every sequence $\{ \epsilon_n \}$
of positive numbers, there exists a sequence $\{ f^+_n \}$ of functions in
$\mathcal{M}_\F$, which satisfy the ``Riemann hypothesis'' and such that
$|f_n^+ - f| < \epsilon_n$ on $E_n$, for each $n$.
\end{theorem}

\begin{proof}
Let $f \in \mathcal{M}_\F$, with $f (s) = \Phi (u)$, and
let $K_n$ and $X_n$ be as above.
From the preceding discussion, we see that we may assume that $f$ has no real
zeros off the critical circle.
Let $Z$ be the set of zeros of $\Phi$ outside the critical circle and in the
upper half-plane.
As above, to a positive number $\delta_n$, we associate an entire function
$g_n$, taking the value $1$ at $1$ and $1/q$, having no zeros except at the
points of $Z$, where $g_n$ has zeros of the same multiplicity as $\Phi$.
Moreover, $|1 - g_n^+| < \delta_n$ on $X_n^+$.
Set
$$
   \nu (u)
 = \frac{1}{g (u) g (1/qu) \overline{g (\overline{u})}
                           \overline{g (1/\overline{qu})}}.
$$
Given $\epsilon_n > 0$, we may choose $\delta_n$ sufficiently small, such that
$$
   |\nu \Phi - \Phi| < \epsilon_n
$$
on $K_n$.
By Lemma \ref{nu}, we get $\nu (u (s)) f(s) \in \mathcal{M}_\F$.
Set $f^+ (s) = \nu (u (s)) f(s)$ and set $E_n = u^{-1} (K_n)$.
Then, $f^+$ satisfies the conclusion of the theorem.

As in the previous theorem, the lemmas involved allow much freedom, so it is
possible to assure that the approximating function is different from $f$.
Thus, the approximation is not trivial.
\end{proof}

\begin{corollary}
For every $f \in \mathcal{M}_\F$ (and in particular for $\zeta_\F$), let $Z_f$
be the zeros of $f$ off the critical axis.
There is an increasing sequence of closed sets
   $E_1 \subset E_2 \subset \ldots$
with $\cup E_n = \C$, and a sequence of functions $\{ f_n \}$ in
$\mathcal{M}_\F$, $f_n \not= f$, which satisfy the ``Riemann hypothesis''
such that
$$
   \lim_{n \rightarrow \infty} \max_{s \in E_n} |f_n (s) - f (s)| = 0.
$$
In particular,
$$
   f_n (s) \rightarrow f (s),
 \,\, \mbox{for every} \,\, s \in \C.
$$
\end{corollary}

{\it Acknowledgements\,}
This research was done while the first author was visiting the Universit\"{a}t
Potsdam and was supported by DFG (Deutschland) and
                             NSERC (Canada).
We are grateful to J\"{o}rn Steuding,
                   Xavier Xarles and
                   Peter Roquette,
who pointed out some errors in an earlier version and made helpful comments.

\newpage

\bibliographystyle{amsplain}

\end{document}